\documentclass[12pt]{amsart}
\usepackage{a4wide}
\usepackage[utf8]{inputenc}
\usepackage{tikz}
\usepackage[backend=biber,style=alphabetic,maxbibnames=99]{biblatex}  
\usepackage{multirow}
\usepackage{textcomp}
\usepackage{csquotes} 
\usepackage{listings}
\usepackage{hyperref}
\hypersetup{colorlinks=false}
\usepackage{amssymb}
\usepackage{enumitem}

\newtheorem{theorem}{Theorem}[section]

\newtheorem{proposition}[theorem]{Proposition}
\newtheorem{lemma}[theorem]{Lemma}

\theoremstyle{definition}
\newtheorem{definition}[theorem]{Definition}
\newtheorem{example}[theorem]{Example}
\newtheorem{remark}[theorem]{Remark}

\author{Antonio J. Guirao}
\address[A. J. Guirao]{Universitat Polit\`ecnica de Val\`encia. Instituto Universitario de Matem\'atica Pura y Aplicada, Camino de Vera, s/n 46022 Valencia (Spain)}
\email{anguisa2@mat.upv.es}

\author{Vicente Montesinos}
\address[V. Montesinos]{Universitat Polit\`ecnica de Val\`encia. Instituto Universitario de Matem\'atica Pura y Aplicada, Camino de Vera, s/n 46022 Valencia (Spain)}
\email{vmontesinos@mat.upv.es}

\author{Andr\'es Quilis}
\address[A. Quilis]{Universitat Polit\`ecnica de Val\`encia. Instituto Universitario de Matem\'atica Pura y Aplicada, Camino de Vera, s/n 46022 Valencia (Spain); and Czech Technical University in Prague, Faulty of Electrical Engineering. Department of Mathematics, Technick\'a 2, 166 27 Praha 6 (Czech Republic)}
\email{anquisan@posgrado.upv.es}

\title[Projectional skeletons and Plichko in Lipschitz-free spaces]{On projectional skeletons and the Plichko property in Lipschitz-free Banach spaces}

\begin{document}
\maketitle
\begin{abstract}
    We study projectional skeletons and the Plichko property in Lipschitz-free spaces, relating these concepts to the geometry of the underlying metric space. Specifically, we identify a metric property that characterizes the Plichko property witnessed by Dirac measures in the associated Lipschitz-free space. We also show that the Lipschitz-free space of all $\mathbb{R}$-trees has the Plichko property witnessed by molecules, and define the concept of retractional trees to generalize this result to a bigger class of metric spaces. Finally, we show that no separable subspace of $\ell_\infty$ containing $c_0$ is an $r$-Lipschitz retract for $r<2$, which implies in particular that $\mathcal{F}(\ell_\infty)$ is not $r$-Plichko for $r<2$.
\end{abstract}
\section{Introduction}

Let $M$ be a complete metric space with a distinguished point $0\in M$. The vector space $\text{Lip}_0(M)=\{f\colon M\rightarrow \mathbb{R}\colon f\text{ is Lipschitz and }f(0)=0\}$ with the norm given by the Lipschitz constant is a Banach space, which is moreover a dual space. Its canonical predual is the subspace $\mathcal{F}(M)=\overline{\text{span}}\{\delta(x)\colon x\in M\}\subset \text{Lip}_0(M)^*$,  where $\delta(x)\colon \text{Lip}_0(M)\rightarrow \mathbb{R}$ is the functional satisfying $\langle f,\delta(x)\rangle=f(x)$ for all $f\in \text{Lip}_0(M)$. This space, introduced by Arens and Eells in \cite{AreEel56}, is also commonly known as the \emph{Lipschitz-free} space associated to $M$, named by Godefroy and Kalton in the seminal article \cite{GodKal03}. The above Dirac map $\delta$ is an isometry embedding, which allows us to regard metric spaces as subsets of their Lipschitz-free space. With this identification, every Lipschitz map $F\colon M\rightarrow N$ between two metric spaces such that $F(0)=0$ can be extended to a linear and bounded operator $\widehat{F}\colon \mathcal{F}(M)\rightarrow Y$ preserving the Lipschitz constant. Since the publication of \cite{GodKal03}, Lipschitz-free spaces have been studied extensively and continue to be a very active area of research in Functional Analysis. We refer to the survey \cite{God15} and the monograph \cite{Wea18Book} for a complete overview of the topic of Lipschitz-free spaces and spaces of Lipschitz functions.

As a continuation of \cite{HajQui22}, in this note, we study the linear projectional structure of Lipschitz-free spaces associated to non-separable metric spaces, which is closely related to the Lipschitz retractional structure of the underlying metric space thanks to the universal extension property that characterizes Lipschitz-free spaces. To this end, we focus on the classical concept of the Plichko property and some related notions.

For a real number $r\geq 1$, a Banach space $X$ is said to be \emph{$r$-Plichko} if there exists a pair $(\Delta,N)$, where $\Delta\subset X$ is a linearly dense subset of $X$ and $N$ is an $r$-norming subspace of $X^*$ such that for every functional $f\in N$, the set
$$ S_\Delta(f)=\{x\in \Delta\colon \langle f,x\rangle \neq 0\}$$
is countable. We say that the pair $(\Delta,N)$ is a \emph{witness} of the Plichko property in $X$. Since $N$ is essentially determined by the linearly dense set $\Delta$, we sometimes say, equivalently, that $X$ is Plichko witnessed by the set
$\Delta$.
 
This concept has been deeply studied in the theory of non-separable Banach spaces. Plichko Banach spaces enjoy geometric and structural properties such as the existence of a LUR renorming, a strong Markushevich basis, or the Separable Complementation Property (SCP). We refer to the survey \cite{Kal00} and the monograph \cite{HajMonVanZiz08} for a complete introduction to this property.

We are also interested in the characterization of the Plichko property in terms of the following concept due to Kubi\'{s}:

\begin{definition}[\cite{Kub09}]
    \label{skeletons}
    Let $X$ be a Banach space. A \emph{projectional skeleton} on $X$ is a family $\{P_s\}_{s\in\Gamma}$ of bounded linear projections on $X$ indexed by a directed partially ordered set $\Gamma$, such that the following conditions hold:
    \begin{itemize}
        \item[(i)] $P_sX$ is separable for all $s\in\Gamma$.
        \item[(ii)] $P_sP_t=P_tP_s=P_s$ whenever $s,t\in \Gamma$ and $s\leq t$.
        \item[(iii)] If $(s_n)_n$ is an increasing sequence of indices in $\Gamma$, then $s=\sup_{n\in\mathbb{N}}s_n$ exists and $P_sX=\overline{\bigcup_{n\in\mathbb{N}}P_{s_n}X}$.
        \item[(iv)] $X=\bigcup_{s\in\Gamma}P_sX$.
    \end{itemize}
    
    If $r\geq 1$, we say that an \emph{$r$-projectional skeleton} is a projectional skeleton where every projection has norm less than or equal to $r$. 
    We say that a projectional skeleton is commutative if $P_sP_t=P_tP_s$ for all $s,t\in\Gamma$, regardless of whether they are comparable or not.
\end{definition}

Indeed, it is proven in \cite{Kub09} that a Banach space is $r$-Plichko if and only if it admits a commutative $r$-projectional skeleton.

We can define an analogous concept to projectional skeletons in metric spaces using Lipschitz retractions:
\begin{definition}
    \label{Lip_skeletons}
Let $M$ be a metric space. A \emph{Lipschitz retractional skeleton} on $M$ is a family $\{R_s\}_{s\in\Gamma}$ of Lipschitz retractions on $M$ indexed by a directed partially ordered set $\Gamma$, such that the following conditions hold:
\begin{itemize}
    \item[(i)] $R_s(M)$ is a separable subset of $M$ for every $s\in \Gamma$.
    \item[(ii)] If $s,t\in \Gamma$ such that $s\leq t$, then $R_s\circ R_t=R_t\circ R_s= R_s$.
    \item[(iii)] If $(s_n)_{n\in\mathbb{N}}$ is a totally ordered sequence in $\Gamma$, then $s=\text{sup}_{n\in\mathbb{N}} s_n$ exists ($\Gamma$ is $\sigma$-complete) and moreover $R_s(M)=\overline{\bigcup_{n\in\mathbb{N}}R_{s_n}(M)}$.
    \item[(iv)] $M=\bigcup_{s\in\Gamma} R_s(M)$.
\end{itemize}
Given $r\geq 1$, an $r$-Lipschitz retractional skeleton is a Lipschitz retractional skeleton with Lipschitz constant uniformly bounded by $r$. A Lipschitz retractional skeleton is \emph{commutative} is $R_s\circ R_t= R_t\circ R_t$ for every $s,t\in \Gamma$, regardless of whether they are comparable or not.
\end{definition}

The linearization property of Lipschitz-free spaces implies that if a metric space admits a (commutative) Lipschitz retractional skeleton, the associated Lipschitz-free space has a (commutative) projectional skeleton. This fact, together with the characterization of the Plichko property in terms of commutative projectional skeletons, shows that the Lipschitz-free space associated to a Plichko space is again Plichko (see Corollary 2.9 in \cite{HajQui22}).

In brief, metric spaces with a rich Lipschitz retractional structure yield Lipschitz-free spaces with a rich linear projectional structure (the converse is not always true, as we discuss in section 5). It is unknown whether there exists a metric space whose Lipschitz-free space does not have Plichko. It is also open whether every Lipschitz-free space has the SCP.

In section 2, we characterize which metric spaces $M$ yield a Lipschitz-free space with the Plichko property witnessed by a subset of $\delta(M)\subset \mathcal{F}(M)$. We show that in this special case, the linear projectional structure of the Lipschitz-free space is strong enough to induce a Lipschitz retractional structure in the underlying space. We offer several equivalent properties, some relating to the structure of the Lipschitz-free space and some relating to the geometry of the underlying metric space. 

In section 3, we study the Lipschitz retractional structure of $\mathbb{R}$-trees (or metric trees). This class of metric spaces has been studied in the context of Lipschitz-free spaces in recent years, as it was shown in \cite{God10} by Godard that the Lipschitz-free space of a complete metric space $M$ is isometric to a subspace of $L_1$ if and only if $M$ is isometric to a subset of a metric tree. See \cite{AliPetPro21} as well for a more recent characterization of metric spaces whose Lipschitz-free space is isometric to a subspace of $\ell_1$ as those which are isometric to subsets of separable $\mathbb{R}$-trees where the length measure of the closure of the set of branching points is $0$. We show in our note that the Lipschitz-free space of every $\mathbb{R}$-tree is $1$-Plichko witnessed by a set of molecules in $\mathcal{F}(M)$. 

In section 4, we introduce the concept of \emph{retractional trees} in metric spaces generalizing the result of the previous section to a bigger class of metric spaces. 

Finally, in section 5, we summarize the relationship between the linear structure of Lipschitz-free spaces associated to Banach spaces and the Lipschitz retractional structure of the underlying space, and we use this connection to show that $\mathcal{F}(\ell_\infty)$ is not $r$-Plichko for $r<2$.

We finish the introduction by setting the notation: Unless stated otherwise, every metric space considered in this note is assumed to be complete and containing a distinguished point $0$. Given a metric space $M$, a point $p\in M$ and a positive number $r>0$, we will denote by $B(p,r)$ the open ball of radius $r$ centered at $p$. Given two different points $p,q\in M$, we will denote by $m_{p,q}$ the (elementary) molecule $\frac{\delta(p)-\delta(q)}{d(p,q)}\in \mathcal{F}(M)$.

\section{Plichko property witnessed by Dirac measures}

Let us introduce the main result of this section:

\begin{theorem}
    \label{equivalencePlichkoDirac}
    Let $M$ be a metric space with distinguished point $0\in M$, and let $\lambda\geq 1$. The following statements are equivalent:

    \begin{itemize}
        \item[(i)] For all $p\in M$ and for all $r<\frac{1}{\lambda}$, the ball $B(p,r\cdot d(p,0))$ is separable. 
        \item[(ii)] $\mathcal{F}(M)$ is $\lambda$-Plichko witnessed by a subset of $\delta(M)$.
        \item[(iii)] $\mathcal{F}(M)$ admits a commutative $\lambda$-projectional skeleton $\{P_s\}_{s\in\Gamma}$ such that $P_s(\delta(p))\in \{0,\delta(p)\}$ for all $p\in M$.
        \item[(iv)] $M$ admits a commutative $\lambda$-Lipschitz retractional skeleton $\{R_s\}_{s\in\Gamma}$ such that $R_s(p)=\{0,p\}$ for all $p\in M$.
        \item[(v)] The closed subspace $\{f\in\text{Lip}_0(M)\colon \text{supp}(f)\text{ is separable}\}$ is a $\lambda$-norming subspace of $\text{Lip}_0(M)$.
    \end{itemize}
\end{theorem}
\begin{proof}
    Write $S_0(M)=\{f\in\text{Lip}_0(M)\colon \text{supp}(f)\text{ is separable}\}$ for the rest of the proof.
    
    Let us start by proving that $(i)$ and $(v)$ are equivalent. Indeed, the proof that $(i)$ implies $(v)$ is already included in \cite[Proposition 2.5]{HajQui22}. To show the converse, suppose by contradiction that $S_0(M)$ is a $\lambda$-norming subspace of $\text{Lip}_0(M)$ and that there exists a point $p_0\in M$ different from $0$ and $0<r_0<\frac{1}{\lambda}$ such that the open ball $B(p_0,r_0\cdot d(p_0,0))$ is non-separable.

    Consider the function $f_0\in \text{Lip}_0(\{0,p_0\})$ defined by $f_0(0)=0$ and $f_0(p_0)=d(p_0,0)$. This function is clearly $1$-Lipschitz, and thus by Kalton's Lemma 3.3 in \cite{Kal04}, choosing $\varepsilon_0>0$ such that $r_0\cdot(\lambda+\varepsilon_0)<1$ there exists a $(\lambda+\varepsilon_0)$-Lipschitz function $g_0$ in $S_0(M)\subset \text{Lip}_0(M)$ such that $g_0(0)=0$ and $g_0(p_0)=d(p_0,0)$. Since we are assuming that $B(p_0,r_0\cdot d(p_0,0))$ is non-separable, there must exist a point $x_0\in M$ with $d(p_0,x_0)< r_0\cdot d(p_0,0)$ such that $g_0(x_0)=0$. However, this implies that
    \begin{align*}
        |g_0(p_0)|&=|g_0(p_0)-g_0(x_0)|\leq (\lambda+\varepsilon_0)d(p_0,x_0)\\
        &<r_0\cdot(\lambda+\varepsilon_0)d(p_0,0)< d(p_0,0),
    \end{align*}
    which contradicts the choice of $g_0$.

    We continue with $(i)$ implies $(ii)$ (This proof is essentially contained in \cite{HajQui22} as well. Still, we include it here for completeness): If a complete metric space $M$ satisfies property $(i)$, we have just proven that the subspace $S_0(M)$ is a $\lambda$-norming subspace of $\text{Lip}_0(M)$. Hence, to show $(ii)$ it is enough to find a dense subset $D$ of $M$ such that the intersection of $D$ with every separable subset of $M$ is countable again. Indeed, it is straightforward to check that in that case, the Lipschitz-free space $\mathcal{F}(M)$ is $\lambda$-Plichko witnessed by $(\delta(D), S_0(M))$.

    To find a dense subset $D$ with this property, we use Zorn's Lemma to obtain a maximal family in the set 
    \begin{align*}
        Z = \big\{\{A_i\}_{i\in I}\colon &A_i\text{ is non-empty, open and separable,}\\
        & \text{and }A_i\cap A_j =\emptyset \text{ for all }i\neq j\in I\big\}.
    \end{align*}
    Denote by $F_0=\{A_i\}_{i\in I}$ such a maximal family in $Z$. Since every point in $M$ different from $0$ has a separable neighborhood, by maximality, we have necessarily that $M=\overline{\bigcup_{i\in I}A_i}\cup \{0\}$. Consider now $D_i$ a countable dense subset of $A_i$ for each $i\in I$. Define $D=\bigcup_{i\in I}D_i$. It is straightforward now to check that $D$ is dense in $M$ and that the intersection of $D$ with any separable subset of $M$ is countable.

    $(ii)$ implies $(iii)$: Write $\Delta\subset \delta(M)$ to denote the set witnessing the $\lambda$-Plichko property in $\mathcal{F}(M)$. Then $S_0(M)$ is countably supported in $\Delta$, so by Proposition 21 in \cite{Kub09}, the Lipschitz-free space $\mathcal{F}(M)$ admits a commutative $\lambda$-projectional skeleton $\{P_s\}_{s\in \Gamma}$ which generates $S_0(M)$. Additionally, since $\Delta$ is a subset of $\delta(M)$ which is linearly dense in $\mathcal{F}(M)$, there exists a dense subset $D$ of $M$ such that $\Delta =\delta(D)$. Now, by Corollary 20 in \cite{CorCutSom22} we may assume that $P_s(\delta(p))\in \{0,\delta(p)\}$ for all $p\in D$ and for all $s\in \Gamma$. By density of $D$ in $M$ and continuity of the projections in the projectional skeleton, we conclude that $P_s(\delta(p))\in \{0,\delta(p)\}$ for all $p\in M$ and all $s\in \Gamma$.

    $(iii)$ implies $(iv)$: Suppose that $\mathcal{F}(M)$ admits a commutative $\lambda$-projectional skeleton $\{P_s\}_{s\in\Gamma}$ such that $P_s(\delta(p))\in \{0,\delta(p)\}$ for all $p\in M$. For every $s\in \Gamma$, the image of the subset $\delta(M)\subset \mathcal{F}(M)$ is contained in $\delta(M)$. Hence, when restricting $P_s$ to the subset $\delta(M)$, we obtain a map $P_{s|\delta(M)}\colon \delta(M)\rightarrow\delta(M)$ which is $\lambda$-Lipschitz. Since the map $\delta\colon M\rightarrow \delta(M)$ is an isometry, this restriction induces a $\lambda$-Lipschitz map 
    $$ R_s\colon M\rightarrow M$$
    for every $s\in \Gamma$, defined by $R_s(p)=\delta^{-1}(P_s(\delta(p)))$. For every $s\in \Gamma$, the map $R_s$ is a retraction onto $\delta^{-1}(P(\delta(M)))$, which is a closed separable subset of $M$. It is direct to check that $\{R_s\}_{s\in \Gamma}$ is a commutative $\lambda$-Lipschitz retractional skeleton in $M$, and it clearly satisfies that $R_s(p)\in \{0,p\}$ for all $p\in M$. 

    $(iv)$ implies $(v)$: Suppose that $M$ admits a commutative $\lambda$-Lipschitz retractional skeleton $\{R_s\}_{s\in\Gamma}$ such that $R_s(p)=\{0,p\}$ for all $p\in M$. We show that $S_0(M)$ is a $\lambda$-norming subspace of $M$ once more by using Kalton's Lemma in \cite{Kal04}. Suppose that $A\subset M$ is a finite set and that $F\in \text{Lip}_0(A)$ is a $1$-Lipschitz function defined in this set. The set $A$ is finite and thus separable, so there exists $s_0\in \Gamma$ such that $S_{s_0}=R_{s_0}(M)$ is a separable set containing $A$. By McShane's extension theorem, we may find another $1$-Lipschitz function $\overline{F}\in \text{Lip}_0(S_{s_0})$ such that $\overline{F}_{|A}=F$. 
    
    Define now $f_{s_0}\colon M\rightarrow \mathbb{R}$ by $f_{s_0}(p)= \overline{F}(R_{s_0}(p))$ for all $p\in M$. The map $f_{s_0}$ is $\lambda$-Lipschitz because $R_{s_0}$ is a $\lambda$-Lipschitz retraction. Since $R_{s_0}(p)=0$ if $p\notin  S_{s_0}$ and $\overline{F}(0)=0$, we obtain that the support of $f_{s_0}$ is contained in the separable subset $S_{s_0}$, which implies that $f_{s_0}$ belongs to $S_0(M)$. We conclude that $S_0(M)$ is $\lambda$-norming by Lemma 3.3 in \cite{Kal04}.

    Since the equivalence between $(i)$ and $(v)$ has already been discussed, we finish the proof of the theorem.
\end{proof}

The geometric condition in $(i)$ of the previous theorem is simple enough to allow us to construct many non-separable metric spaces whose Lipschitz-free spaces have the Plichko property witnessed by Dirac measures. The following example shows that such metric spaces can be found in the Banach spaces $\ell_p(\Gamma)$ for every $1\leq p\leq \infty$ and every uncountable cardinal $\Gamma$.
\begin{example}
\label{ExampleOfPlichkoDeltas1}
    Let $p\in [1,\infty]$ and let $\Gamma$ be uncountable. The Banach space $\ell_p(\Gamma)$ contains a complete metric space $M_p$ of density character $\Gamma$ such that $\mathcal{F}(M_p)$ is $1$-Plichko witnessed by Dirac measures.

    To show this, fix $p\in [1,\infty]$. For each $\gamma\in \Gamma$, denote $E_\gamma =[0,e_\gamma]= \{(x_\nu)_{\nu\in\Gamma}\in \ell_p(\Gamma)\colon x_{\gamma}\in [0,1],~\text{and } x_{\nu}=0~ \text{for }\nu\neq \gamma\}$. Put $M_p = \bigcup_{\gamma\in\Gamma}E_\gamma$. Then $M_p$ is a complete metric space of density character $\Gamma$, and it is easy to check that for every $x\in M_p$ different from $0$, the open ball $B(x,d(x,0))$ is contained in the separable segment $E_\gamma$ such that $x\in E_\gamma$. Hence, $M_p$ satisfies property $(i)$ for $\lambda=1$ in Theorem \ref{equivalencePlichkoDirac} and we conclude that $\mathcal{F}(M_p)$ is $1$-Plichko witnessed by Dirac measures.
\end{example}
In the previous example, the constructed metric space $M_p$ has a stronger property than the geometric condition we demand: Every point $x$ in $M_p$ is contained in a separable subset $E_x$ such that for every point $y\in E_x$ the open ball $B(y,d(y,0))$ is contained in $E_x$, and for every point $z\notin E_x$ the open ball $B(z,d(z,0))$ is contained in $M\setminus E_x$. Informally, these metric spaces can be seen as the union of (uncountably many) open and closed separable components ``glued'' at the distinguished point $0$ in such a way that the distance between points in two different components is always greater than the distance of each individual point to $0$. This fact may suggest that the geometric condition in $(i)$ of Theorem \ref{equivalencePlichkoDirac} is only satisfied by metric spaces that are very close to being separable in the previous sense. However, we show in the next example that this is not the case:

\begin{example}
\label{ExampleOfPlichkoDeltas2}
    There exists a non-separable metric space $N_2$ isometric to a subset of $\ell_2(\Gamma)$ for uncountable $\Gamma$ containing $e_\gamma$ for all $\gamma\in \Gamma$, whose Lipschitz-free space is $1$-Plichko witnessed by Dirac measures with the following property:
    
    For every $\gamma\in \Gamma$ and every separable subset $S_\gamma$ in $N_2$ containing $e_\gamma$, there exists a point $x\in N_2\setminus S$ such that $e_\gamma$ belongs to the open ball $B(x,d(x,0))$. 

    To construct such a space, consider $M_2\subset \ell_2(\Gamma)$ of Example \ref{ExampleOfPlichkoDeltas1}. Define $N_2 = M_2 \cup \{e_\gamma+e_\nu\colon \gamma\neq \nu\in \Gamma\}$. As before, for every $x\in M_2$, the open ball $B(x,d(x,0))$ in $N_2$ is contained in the separable set $E_\gamma$ such that $x\in E_\gamma$. Additionally, for every pair of indices $\gamma\neq \nu\in \Gamma$, we have that $d(e_\gamma+e_\nu,0)=\sqrt{2}$, and thus the open ball $B(e_\gamma+e_\nu,d(e_\gamma+e_\nu,0))$ in $N_2$ is contained in the union of $E_\gamma\cup E_\nu\cup\{e_\gamma+e_\nu\}$, which is separable as well. The metric space $N_2$ satisfies condition $(i)$ with $\lambda=1$ in Theorem \ref{equivalencePlichkoDirac}, and we have that $\mathcal{F}(N_2)$ is $1$-Plichko witnessed by Dirac measures.

    Now, for every $\gamma\in \Gamma$, given any separable subset $S_\gamma$ in $N_2$ containing $e_\gamma$, there exists $\nu\in \Gamma$ such that $e_\gamma+e_\nu$ does not belong to $S_\gamma$ (since that would contradict the separability of $S_\gamma$). It is simple to verify that the point $e_\gamma$ belongs to the open ball $B(e_\gamma+e_\nu,d(e_\gamma+e_\nu,0))$.
\end{example}

In this last example, we may also intuitively identify separable components whose union forms the whole complete metric space (each set of the form $E_\gamma$ or $E_{\gamma,\nu}$ for every $\gamma,\nu\in \Gamma$). However, as we have proven, given a point $p$ in one component, there may be points in different components closer to $p$ than the value $d(p,0)$. Let us formalize this intuition by characterizing metric spaces with the Plichko property witnessed by Dirac measures in terms of their metric structure:

Let $M$ be a complete metric space, and let $\lambda\geq 1$. We say that a collection $\mathcal{S}$ of subsets of $M$ is a \textit{separable $\lambda$-slab decomposition} if for all $N\in \mathcal{S}$, the set $N\setminus\{0\}$ is an open separable set, and for all $p\in N$ there exists a countable subfamily $\mathcal{S}_p\subset \mathcal{S}$ such that $B(p,\lambda\cdot d(p,0))\cap \bigcup_{N\in\mathcal{S}}N$ is contained in $B(p,\lambda\cdot d(p,0))\cap\bigcup_{N\in \mathcal{S}_p} N$.

We say that a separable $\lambda$-slab decomposition is \textit{total} if $M=\overline{\bigcup_{N\in \mathcal{S}} N}$.

\begin{proposition}
\label{metricstructurechar_PlichkoDirac}
Let $M$ be a complete metric space and let $\lambda\geq 1$. The following statements are equivalent:

\begin{itemize}
    \item[(i)] $\mathcal{F}(M)$ is $\lambda$-Plichko witnessed by a subset of $\delta(M)$.
    \item[(ii)] For every separable $\lambda$-slab decomposition $\mathcal{S}$ of $M$ there exists a total separable $\lambda$-slab decomposition $\mathcal{S}'$ such that $\mathcal{S}\subset \mathcal{S}'$.
    \item[(iii)] $M$ admits a total separable $\lambda$-slab decomposition.
\end{itemize}
\end{proposition}
\begin{proof}
We start by showing that $(i)$ implies $(ii)$. Assume first that $\mathcal{F}(M)$ is $\lambda$-Plichko witnessed by a set $\Delta\subset \delta(M)$, and fix a separable $\lambda$-slab decomposition $\mathcal{S}$. Consider the following family:
$$\Omega := \big\{\mathcal{D}=\{N_i\}_{i\in I}\colon \mathcal{S}\subset \mathcal{D},  ~\mathcal{D}\text{ is a separable $\lambda$-slab decomposition}\big\},  $$
which can be partially ordered by inclusion. Using Zorn's Lemma, we obtain a maximal element $\mathcal{S}'\in \Omega$. Let us prove that $\mathcal{S}'$ is total. 

Indeed, suppose by contradiction that there exists a point $p\in M\setminus\{0\}$ such that $p$ lays outside the closure of the set $\bigcup_{N\in \mathcal{S}'} N$. Then, applying condition $(i)$ of Theorem \ref{equivalencePlichkoDirac} we can find an open and separable set $A_p$ containing $p$ such that $A_p\cap N=\emptyset$ for all $N\in \mathcal{S}'$. Since for every point $x\in A_p$ the ball $B(x,\lambda\cdot d(x,0))$ is separable and $N\setminus\{0\}$ is open for all $N\in \mathcal{S}'$, it follows that $\mathcal{S}_p=\mathcal{S}'\cup\{A_p\cup\{0\}\}$ is a separable $\lambda$-slab decomposition which contains $\mathcal{S}$, contradicting the maximality of $\mathcal{S}'$ in $\Omega$. 

Since $\mathcal{S}=\{0\}$ is trivially a separable $\lambda$-slab decomposition, it is clear that $(ii)$ implies $(iii)$.

Finally, to show that $(iii)$ implies $(i)$, suppose there exists a total separable $\lambda$-slab decomposition $\mathcal{S}$ in $M$. We will prove that condition $(i)$ in Theorem \ref{equivalencePlichkoDirac} holds for the dense subset $\bigcup_{N\in \mathcal{S}}N$, which is enough to prove that $M$ is $\lambda$-Plichko witnessed by Dirac measures. 

Consider a separable set $N\in \mathcal{S}$ and fix $p\in N\setminus\{0\}$. First, since $\mathcal{S}$ is total, it follows that 
$$B(p,\lambda\cdot d(p,0))=\overline{\bigcup_{N\in\mathcal{S}}N\cap B(p,\lambda\cdot d(p,0))}.$$
Next, since $\mathcal{S}$ is a separable $\lambda$-slab decomposition, there exists a countable subfamily $\mathcal{S}_0\subset S$ such that 
$$ B(p,\lambda\cdot d(p,0))= \overline{\bigcup_{N\in\mathcal{S}_0}N\cap B(p,\lambda\cdot d(p,0))}.$$
This shows that $B(p,\lambda\cdot d(p,0))$ is the closure of a countable union of separable sets, thus separable itself. 
\end{proof}

We finish this section by mentioning an example of a complete metric space whose Lipschitz-free space is $1$-Plichko, but which fails to have a subset of Dirac measures witnessing such property: Consider any $\mathbb{R}$-tree with at least two branching points with uncountably many branches at each one of them. Then, whichever choice of the distinguished point we make, there will be a branching point with no separable neighborhood. This implies that condition $(i)$ in Theorem \ref{equivalencePlichkoDirac} does not hold, and the conclusion follows. However, as we will see in the next section, every $\mathbb{R}$-tree does have the $1$-Plichko property, though witnessed by a set of molecules instead. 

\section{The $1$-Plichko property in the Lipschitz-free space of an $\mathbb{R}$-tree}

In this section, we prove that the Lipschitz-free space associated to any $\mathbb{R}$-tree has the $1$-Plichko property witnessed by molecules. 

Let us start by formally defining an $\mathbb{R}$-tree. A metric space $(T,d)$ is an \emph{$\mathbb{R}$-tree} if for every pair of points, $x\neq y\in T$ there exists a unique arc $[x,y]\subset T$, and moreover, the arc $[x,y]$ is isometric to the real line segment $[0,d(x,y)]$. 

A subset of $T$ will be called an \emph{$\mathbb{R}$-subtree} if it is an $\mathbb{R}$-tree which contains the distinguished point $0\in T$ 

Given any $\mathbb{R}$-tree $T$, and given any closed $\mathbb{R}$-subtree $S$ of $T$, it holds that for every point $t\in T$ there exists a unique point $s\in S$ such that the distance from $t$ to $S$ is attained at only the point $s$. This defines a retraction $P_S\colon T\rightarrow S$, where $P_S(t)$ is the unique point in $S$ such that $d(t,S)=d(t,P_S(t))$, called the \emph{metric projection of $T$ onto $S$}. It is known that $P_S$ is a $1$-Lipschitz retraction for every closed $\mathbb{R}$-subtree $S$ of $T$. We will also use the fact that given two closed $\mathbb{R}$-subtrees $A$ and $B$ of $T$ it holds that $P_A\circ P_B =P_B\circ P_A=P_{A\cap B}$. Notice that this implies that given two closed $\mathbb{R}$-subtrees $A$ and $B$ of $T$, and a point $p\in A$, the point $P_B(p)$ belongs to $A\cap B$. 

\begin{theorem}
\label{trees_Plichko_with_molecules}
    Let $T$ be an $\mathbb{R}$-tree. Then $\mathcal{F}(T)$ is $1$-Plichko witnessed by a pair $(\Delta, N)$ where $\Delta$ is a linearly dense set of molecules.
\end{theorem}
\begin{proof}
    Let $D\subset T$ be a dense set in $T$ such that $D\cap [0,p]$ is countable for every $p\in T$. We will define inductively an increasing sequence of families $\{G_\alpha\}_{\alpha\leq \kappa}$, where $\kappa$ is a big enough cardinal and $G_{\alpha}$ is a family of closed separable $\mathbb{R}$-subtrees of $T$ for every $\alpha\leq\kappa$, such that $T =\bigcup_{S\in G_{\kappa}}S$ and the following properties are satisfied for every $\alpha\leq\kappa$:
    \begin{itemize}
        \item[(a)] For every $H\in G_\alpha$, the family $\{S\in G_\alpha\colon S\subset H\}$ is totally ordered for inclusion. That is: for every $S_1,S_2\in G_\alpha$ with $S_1,S_2\subset H$, either $S_1\subset S_2$ or $S_2\subset S_1$.  
        \item[(b)] $G_\alpha$ is closed for finite intersections. Moreover, for every pair $H_1,H_2\in G_\alpha$ with $H_i\in G_{\beta_i}$ for some $\beta_i\leq \alpha$, and such that $H_1\cap H_2\subsetneq H_i$ for $i=1,2$, there exists $\beta<\min\{\beta_1,\beta_2\}$ such that $H_1\cap H_2\in G_\beta$. In other words, any two incomparable sets in $G_\alpha$ meet at a set which belongs to a strictly lower generation than both sets.
        \item[(c)] For every family $\{S_n\}_{n\in\mathbb{N}}$ of $\mathbb{R}$-subtrees such that $S_n\in G_{\alpha}$ and $S_n\subset S_{n+1}$ for all $n\in\mathbb{N}$, the set $\overline{\bigcup_{n\in\mathbb{N}}S_n}$ is in $G_\alpha$.
    \end{itemize}
    
    Put $G_0 = \left\{\{0\}\right\}$, which trivially satisfies the assumptions. Suppose we have constructed $G_\beta$ for all $\beta$ smaller than a limit ordinal $\alpha$. 
    
    Then put
    $$ G_\alpha = \bigcup_{\beta<\alpha}G_\beta\cup\left\{\overline{\bigcup_{n\in\mathbb{N}} S_n}\colon\text{ where }S_n\in G_{\beta}\text{ for some }\beta_n<\alpha,\text{ and }S_{n}\subset S_{n+1}\text{ for all }n\in\mathbb{N}\right\}.$$
    It is direct that condition $(c)$ holds in this case. Condition $(a)$ follows easily from the inductive hypothesis. 
    
    Finally, while it is clear that $G_\alpha$ is closed for finite intersections, to show the second part of $(b)$, consider two sets $H_1$ and $H_2$ in $G_\alpha$ such that $H_i$ is not a subset of $H_j$ for $i\neq j\in \{1,2\}$. For $i\in\{1,2\}$, we may assume that $H_i = \bigcup_{n\in\mathbb{N}}S^i_n$ with $S^i_n\in G_{\beta^i_n}$ for some $\beta_n^i<\alpha$ and $S^i_n\subset S^i_{n+1}$ for all $n\in\mathbb{N}$. Then, there exists $n_0\in\mathbb{N}$ such that $S^i_{n_0}$ is not a subset of $S^{j}_{n_0}$ for $i\neq j\in\{1,2\}$. Applying $(a)$ to $H_1$, we obtain that $H_1\cap S^2_{n_0}= S^1_{n_0}\cap S^2_{n_0}$. Similarly, we obtain that $H_1\cap H_2 = S^1_{n_0}\cap S^2_{n_0}$. Condition $(b)$ for the pair $S^1_{n_0}$ and $S^2_{n_0}$ yields the desired result.
    
    For the successor ordinal case, assume $G_\alpha$ is defined for a given ordinal $\alpha$. For every $p\in \bigcup_{S\in G_\alpha}S$ we may define the ordinal $\text{height}(p)$ as the least ordinal $\beta\leq \alpha$ such that $p\in H$ for some $H\in G_\beta$. Using minimality, such a set is unique by the second part of condition $(b)$. Hence, we may define $H(p)\in G_{\text{height}(p)}$ as the (unique) smallest $\mathbb{R}$-subtree in $G_\alpha$ containing $p$. 
    
    If $T=\bigcup_{S\in G_\alpha} S$ we stop the inductive process with $\kappa = \alpha$. Otherwise, consider the collection of families of separable $\mathbb{R}$-subtrees given by:
    \begin{align*}
        \Omega_\alpha = \big\{\{ H_i\}_{i\in I}\colon & H_i \text{ is a separable $\mathbb{R}$-subtree and there exists a family }\{S_i\}_{i\in I}\subset G_\alpha\text{ such that}\\
        &\text{for every }i\in I,\text{ the set } S_i\text{ is strictly contained in }H_i\\
        &\text{and }H_i\cap S = S_i\cap S\text{ for all }S\in G_\alpha,\\
        &\text{and such that }H_i\cap H_j = S_i\cap S_j\text{ for }i\neq j\in I\big\}.
    \end{align*} 
    Let us show that there always exists a nonempty set $\mathcal{H}_\alpha\in \Omega_\alpha$. Consider any $p\notin \bigcup_{S\in G_\alpha}S$, and write $P_0 = \sup_{S\in G_\alpha}\{S\cap[0,p]\}$ (here we are taking the maximum with respect to the natural order in $T$ given by: For any $p,q\in T$, we say that $p\leq q$ if $[0,p]\subset [0,q]$). We will prove that there exists an $\mathbb{R}$-subtree $S_0\in G_\alpha$ such that $S_0\cap [0,p]= [0,P_0]$.
    
    Consider a sequence $\{p_n\}_{n\in\mathbb{N}}$ in $\bigcup_{S\in G_\alpha}S$ such that $p_n\leq p_{n+1}$ and $p_n\rightarrow P_0$. The set $H(p_n)\in G_{\text{height}(p_n)}$ is in $G_\alpha$ for every $n\in\mathbb{N}$. Note that $p_n$ belongs to $H(p_m)$ for all $m\geq n$. If there exists $n\in\mathbb{N}$ such that $H(p_n)$ contains every $p_m$ for all $m\geq n$, then the limit point $P_0$ belongs to the closed $\mathbb{R}$-subtree $H(p_n)$, and we obtain that $H(p_n)\cap [0,p]=[0,P_0]$ as desired. 
    
    If not, we can assume by passing to a subsequence that $H(p_m)$ is not a subset of $H(p_n)$ for every pair $n,m\in\mathbb{N}$ with $n<m$. This implies that $H(p_n)$ is a subset of $H(p_m)$ for $n,m\in\mathbb{N}$ with $n<m$, since otherwise the second part of condition $(b)$ shows that $H(p_n)\cap H(p_m)$ belongs to $G_{\beta}$ for some $\beta<\text{height}(p_n)$, which contradicts the minimality of $\text{height}(p_n)$. Hence, applying condition $(c)$ to the increasing sequence $\{H(p_n)\}_{n\in\mathbb{N}}$ yields the sought $\mathbb{R}$-subtree in $G_{\alpha}$.
     
    Once we have an $\mathbb{R}$-subtree $S_0\in G_\alpha$ with $S_0\cap [0,p]= [0,p_0]$, it is straightforward to see that the singleton $\{S_{0}\cup[0,p]\}$ belongs to $\Omega_\alpha$. Choose any nonempty $\mathcal{H}_\alpha\in \Omega_\alpha$ and put $G_{\alpha+1}=G_\alpha\cup \mathcal{H}_\alpha$. 

    It is straightforward to check that $G_{\alpha+1}$ satisfies the inductive hypothesis. Since for every ordinal $\alpha$ the set $\bigcup_{S\in G_{\alpha} }S$ is strictly contained in $\bigcup_{S\in G_{\alpha+1}}S$, there exists an ordinal $\kappa$ such that $T =\bigcup_{S\in G_{\kappa}}S$. This finishes the inductive construction.

    Once the induction is finished, for every $p\in T$ we keep the definition of $\text{height}(p)$ and $H(p)\in G_{\text{height}(p)}$. Now, for every $p\in T$ such that $\text{height}(p)$ is a successor ordinal $\alpha+1$, we wish to define $S(p)\in G_\alpha$ as the unique set in $G_\kappa$ that contains every set $S$ in $G_\kappa$ strictly contained in $H(p)$. In order to show that such a set exists and is unique, consider the family $\{S\in G_\alpha\colon S\subset H(p)\}$, which is well ordered by $(a)$. Since $\overline{\bigcup\{S\in G_\alpha\colon S\subset H(p)\}}$ is contained in the separable set $H(p)$, there exists a sequence $\{S_n\}_{n\in\mathbb{N}}$ such that 
    $$\overline{\bigcup_{n\in\mathbb{N}}S_n}=\overline{\bigcup\{S\in G_\alpha\colon S\subset H(p)\}}.$$
    Condition $(c)$ now implies that the set $S(p)=\overline{\bigcup_{n\in\mathbb{N}}S_n}$ belongs to $G_\alpha$, which clearly has the desired properties.
    
    Denoting by $P_S\colon T\rightarrow S$ the metric projection onto any $\mathbb{R}$-subtree $S$, we consider now the family of molecules given by:
   \begin{align*}
       \Delta = \big\{m_{p,P_{S(p)}(p)}\colon &p\in D \text{ with }H(p)\in G_{\alpha+1}\text{ for some }\alpha<\kappa\big\}.
   \end{align*} 

   We start by showing that $\Delta$ is linearly dense in $\mathcal{F}(T)$. It is enough to prove that $\delta(p)$ belongs to the closed linear span of $\Delta$ for every $p\in T$. We work by induction on the height of $p$. If $\text{height}(p)=0$, then $p=0$ and the conclusion follows trivially. 
   
   Assume that $\text{height}(p)$ is a limit ordinal $\alpha$, and that $\delta(q)\in \overline{\text{span}}(\Delta)$ for all $q\in M$ with $\text{height}(q)<\alpha$. By definition of $G_\alpha$ there exists a sequence of ordinals $\{\beta_n\}_{n\in\mathbb{N}}$ with $\beta_n<\alpha$ for all $n\in\mathbb{N}$, and a sequence of sets $\{S_{\beta_n}\}_{n\in\mathbb{N}}$ with $S_{\beta_n}\in G_{\beta_n}$ such that $H(p)=\overline{\bigcup_{n\in\mathbb{N}} S_{\beta_n}}$. Since $p\in H(p)$, it follows by inductive hypothesis that $p\in \overline{\text{span}}(\Delta)$.

    Assume now that $\text{height}(p)=\alpha+1$ for some ordinal $\alpha$, and that $\delta(q)\in \overline{\text{span}}(\Delta)$ for all $q\in M$ with $\text{height}(q)=\alpha$. Suppose first that $p\in D$. Since $P_{S(p)}(p)$ belongs to $P_{S(p)}(M)$ and $\text{height}(q)=\alpha$ for all $q\in S(p)$, we have by inductive hypothesis that $\delta(P_{S(p)})\in \overline{\text{span}}(\Delta)$. Therefore, we obtain that
    $$\delta(p)=d\left(p,P_{S(p)}(p)\right)m_{p,P_{S(p)}(p)}+\delta(P_{S{p}})\in\overline{\text{span}}(\Delta). $$ 
    Having proven the claim for all $p\in D$, it follows by density of $D$ that $\delta(p)\in \overline{\text{span}}(\Delta)$ for all $p\in M$ with $\text{height}(p)=\alpha+1$. Hence, we conclude that $M=\overline{\text{span}}(\Delta)$. 
   
   Finally, let $N=\{f\in \text{Lip}_0\colon f\text{ is countably supported in }\Delta\}$. We will show that if $C$ is a separable subset of $T$ containing the distinguished point $0$, and $\varphi$ is a $1$-Lipschitz function in $\text{Lip}_0(C)$, there exists a $1$-Lipschitz function $f\in N$ that extends $\varphi$. This implies, by Kalton's Lemma, that $N$ is $1$-norming.

   Consider such $C$ and $\varphi$. Since $C$ is separable, there exists a countable subfamily $\{S_n\}_{n\in\mathbb{N}}$ of $G_\kappa$ such that $C\subset S^*=\bigcup_{n\in\mathbb{N}} S_n$. Using McShane's theorem, we can extend $\varphi$ to a function $F\in\text{Lip}_0(S^*)$ which is $1$-Lipschitz. Now, define $f(p)=F(P_{S^*}(p))$ for all $p\in T$, where $P_{S^*}$ is the metric projection onto the separable $\mathbb{R}$-subtree $S^*$. We will prove that $f$ is finitely supported in $\Delta$ and thus belongs to $N$.

   Let $p\in D$, with $H(p)\in G_{\alpha+1}$ for some ordinal $\alpha<\kappa$. Suppose first that $H(p)\nsubseteq S^*$. Then $P_{S^*}(p)\neq p$. By commutativity of the metric projections, we have that $P_{S^*}(p)\in H(P_{S^*}(p))\cap H(p)$. This implies that $H(P_{S^*}(p))$ is contained in $H(p)$, and the inclusion is strict since $H(P_{S^*}(p))\subset S^*$. We conclude that $H(P_{S^*}(p))\subset S(p)$. Therefore, again by commutativity of the metric projections, we obtain that $P_{S^*}(P_{S(p)}(p))=P_{S^*}(p)$, which implies that $\langle f, m_{p, P_{S(p)}(p)}\rangle=0$.

    Now, if $H(p)\subset S^*$, then $p\in S^*$. Since $S^*$ is separable and $D\cap S_n$ is countable for all $n\in\mathbb{N}$, the set $\{p\in D\colon H(p)\subseteq S^*\}$ is countable. We conclude that $f$ is countably supported in $\Delta$. Hence, $f\in N$ and we obtain that $N$ is $1$-norming.
\end{proof}

\section{Lipschitz-retractional trees}

We will generalize the previous ideas in the proof of Theorem \ref{trees_Plichko_with_molecules} to obtain the same result for a bigger class of metric spaces. To define this concept, we will use the notion of \emph{tree} in the partially ordered sense: A partially ordered set $(\Gamma,\leq )$ is called a \emph{tree} if for every element $s\in \Gamma$, the set $\{t\in \Gamma\colon t<s\}$ is a well-ordered set. A tree is said to be \emph{rooted} if there is a minimum element in $\Gamma$, called the \emph{root} and usually denoted by $0$, and is called \emph{$\sigma$-complete} if every totally ordered countable subset of $\Gamma$ has a supremum in $\Gamma$. 

For any element $s$ in a tree $\Gamma$, we call $I_s=\{t\in\Gamma\colon t<s\}$ the initial segment generated by $s$. Notice that $I_s$ is well ordered for every $s\in \Gamma$ and is, therefore, order isomorphic to an ordinal number, which we will refer to as the height of $s$. The height of a tree $\Gamma$ is the supremum over the height of all elements of $\Gamma$. 

In a $\sigma$-complete tree, given two points $s$ and $t$ of countable height, it is direct to show that there exists a unique $s\wedge t\in \Gamma$, defined as the supremum in $\Gamma$ over all points which are lower bounds of both $s$ and $t$. 


Let $M$ be a complete metric space. A family of $\lambda$-Lipschitz retractions $\{R_s\}_{s\in\Gamma}$ on $M$ indexed by a rooted, $\sigma$-complete tree $\Gamma$ such that $\text{height}(s)<\omega_1$ for all $s\in \Gamma$ is called a $\lambda$-\emph{Lipschitz pre-retractional tree} if the following properties are satisfied:

\begin{itemize}
        \item[(i)] $R_0(M)=\{0\}$ and $R_s(M)$ is a separable subset of $M$ for every $s\in \Gamma$.
        \item[(ii)] For every $s,t\in \Gamma$ we have that $R_s\circ R_t= R_t\circ R_s=R_{s\wedge t}$.
        \item[(iii)] If $(s_n)_{n\in\mathbb{N}}$ is a totally ordered sequence in $\Gamma$ with $s=\sup_{n\in\mathbb{N}}s_n$, then $R_s(M)=\overline{\bigcup_{n\in\mathbb{N}}R_{s_n}(M)}$.
        \item[(iv)] $M=\bigcup_{n\in\mathbb{N}}R_s(M)$.   
\end{itemize}

Notice that if $M$ admits a Lipschitz pre-retractional tree $\{R_s\}_{s\in\Gamma}$ such that $\text{height}(s)<\omega_1$ for all $s\in \Gamma$, given a point $p\in M$ we can find an element $s(p)\in \Gamma$ such that $p\in R_{s(p)}(M)$ and such that $s(p)$ is minimal in $\Gamma$ for this property. Moreover, condition $(ii)$ ensures that $s(p)$ is unique. This allows us to define the height of a point $p$ in a Lipschitz pre-retractional tree $\{R_s\}_{s\in\Gamma}$ as the height of $s(p)$ in the tree $\Gamma$.

Now, given a $\sigma$-complete subset $A\subset \Gamma$, we may define a map $R_A\colon M\rightarrow \bigcup_{s\in A}R_s(M)$ given by $R_A(p)=R_{s(p)\wedge A}(p)$. Condition $(ii)$ ensures that this map is well defined, and it clearly is a retraction. We can finally define the structure we will be studying in this section:
\begin{definition}
    Let $M$ be a complete metric space. A family of $\lambda$-Lipschitz retractions $\{R_s\}_{s\in\Gamma}$ on $M$ indexed by a rooted, $\sigma$-complete tree $\Gamma$ such that $\text{height}(s)<\omega_1$ for all $s\in \Gamma$ is called a $\lambda$-\emph{Lipschitz retractional tree} if $\{R_s\}_{s\in\Gamma}$ is a $\lambda$-Lipschitz pre-retractional tree which additionally satisfies the following condition:
    \begin{itemize}
        \item[(v)] For every $\sigma$-complete countable set $A$ of $\Gamma$, the set $R_A(M)$ is closed and the retraction $R_A$ is $\lambda$-Lipschitz.
    \end{itemize}
\end{definition}

In the proof of Theorem \ref{trees_Plichko_with_molecules} we construct a similar structure to a $1$-Lipschitz retractional tree in every $\mathbb{R}$-tree. We are going to show that this construction is enough to obtain the Plichko property in the Lipschitz-free space witnessed by molecules and that it can be applied to metric spaces which are not isometric to an $\mathbb{R}$-tree.

Let us prove first one technical lemma:

\begin{lemma}
    \label{dense_count_tree}
    Let $M$ be a metric space. Suppose $M$ admits a Lipschitz retractional tree $\{R_s\}_{s\in\Gamma}$. Then there exists a dense set $D$ in $M$ such that $D\cap R_s(M)$ is countable for every $s\in\Gamma$.
\end{lemma}
\begin{proof}
    Let $\kappa$ be the height of the tree $\Gamma$. We are going to define by transfinite induction a family of sets $\{D_\alpha\}_{\alpha\leq \kappa}$ with the following properties for every $\alpha<\kappa$:
    \begin{itemize}
        \item[(1).] $D_\alpha$ contains $D_\beta$ for all $\beta<\alpha$.
        \item[(2).] $D_\alpha\cap R_s(M)$ is a countable dense subset of $R_s(M)$ for every $s\in\Gamma$ such that $\text{height}(s)\leq \alpha$.  
    \end{itemize}
    It is clear that once we construct such a family, the set $D=D_\kappa$ will satisfy the conditions of the lemma.

    Put $D_0=\{0\}$. Suppose we have constructed $\{D_\beta\}_{\beta<\alpha}$ for some ordinal $\alpha<\kappa$. If $\alpha$ is a limit ordinal, simply define $D_\alpha=\bigcup_{\beta<\alpha} D_\beta$. It follows from property $(iii)$ of the Lipschitz retractional tree that $D_\alpha$ satisfies condition $(2)$ of the inductive process. 

    If $\alpha$ is not a limit ordinal, then $\alpha=\eta+1$ for some ordinal $\eta<\kappa$. For each $s\in\Gamma$ such that $\text{height}(s)=\eta+1$, consider a countable dense subset $D_s$ of $R_s(M)\setminus \overline{D_\eta}$. Finally, Put
    $$ D_{\eta+1}= D_\eta\cup\bigg(\bigcup \{D_s\colon s\in\Gamma\text{ and }\text{height}(s)=\eta+1\}\bigg).$$
    Clearly, condition $(1)$ is satisfied, and condition $(2)$ follows from the definition of each $D_s$ and condition $(ii)$ in the definition of Lipschitz pre-retractional trees.
\end{proof}

We can now prove the main theorem of the section.
\begin{theorem}
    Let $M$ be a complete metric space. If $M$ admits a $\lambda$-Lipschitz retractional tree, then $\mathcal{F}(M)$ is $\lambda$-Plichko witnessed by a pair $(\Delta,N)$, where $\Delta$ is a subset of molecules of $M$. 
\end{theorem}
\begin{proof}
    Let $\{R_s\}_{s\in\Gamma}$ be a $\lambda$-Lipschitz retractional tree. Using Lemma \ref{dense_count_tree}, choose $D$ to be a dense subset of $M$ such that $D\cap R_{s}(M)$ is countable for every $s\in\Gamma$. 

    Let $\kappa$ be the height of the tree $\Gamma$. Fix a point $p\in D$ such that $\text{height}(p)$ is $\alpha+1$ for some ordinal number $\alpha<\kappa$. Since the initial segment $I_{s(p)}$ is order isomorphic to the ordinal $\alpha+1$, there exists a unique element $t(p)$ in $I_{s(p)}$ with $\text{height}(t(p))=\alpha$. Define now
    $$\Delta =\bigcup \{m_{p,R_{t(p)}(p)}\colon p\in D, ~\text{height}(p)=\alpha+1\text{ for some }\alpha<\kappa\}. $$
    Notice that $\Delta$ is a subset of molecules of $M$. 
    
    With a similar inductive reasoning as in the proof of Theorem \ref{trees_Plichko_with_molecules}, we obtain that $\Delta$ is linearly dense in $\mathcal{F}(M)$.

    Finally, we show that the set $N=\{f\in\text{Lip}_0(M)\colon f\text{ is countably supported in }\Delta\}$ is $\lambda$-norming. We will show that if $S$ is a separable subset of $M$ containing the distinguished point $0$, and $g$ is a $1$-Lipschitz function in $\text{Lip}_0(S)$, there exists a $\lambda$-Lipschitz function $f\in N$ that extends $g$. 

    Consider $S$ and $g$ with such properties. Since $S$ is separable and $M=\bigcup_{s\in\Gamma}R_s(M)$, there exists a $\sigma$-complete and countable subset $A$ of $\Gamma$ such that $S\subset \bigcup_{s\in A}R_s(A)$. Define $f\in\text{Lip}_0(M)$ as $f= g\circ R_A$. Since $R_A$ is $\lambda$-Lipschitz, it follows that $\|f\|_\text{Lip}\leq \lambda$, and clearly $f$ extends $g$ since $R_A$ is a retraction onto $\bigcup_{s\in A}R_s(A)$. It remains to show that $f$ is countably supported in $\Delta$. 
    
    To this end, we will show that for any point $p\in D$ such that $\text{height}(p)$ is a successor ordinal and such that $s(p)\notin A$, we have that $\langle f, m_{p, R_{t(p)}(p)}\rangle =0$. In this situation, since $s(p)\notin A$, we have that $t(p)$ is greater or equal than $s(p)\wedge A$, and thus it follows that $t(p)\wedge A=s(p)\wedge A$ and $R_A(p)=R_A(R_{t(p)}(p))$. This implies that $f(p)=f(R_{t(p)}(p))$, which finishes the proof.
\end{proof}

We finish with an example of a non-separable metric space that is not isometric to an $\mathbb{R}$-tree though it admits a $1$-Lipschitz retractional tree. 
\begin{example}
    Let $\kappa$ be an uncountable set. Let $M=\{0,1\}\cup\{(t,\alpha)\colon t\in (0,1),~\alpha\in \kappa\}$, endowed with the unique distance $d\colon M\times M\rightarrow [0,\infty)$ satisfying
    $$ 
    d(p,q)=
    \begin{cases}
    1,&\text{ if }p=0,~q=1\\
    t,&\text{ if }p=0,~q=(t,\alpha),\text{ for some }\alpha\in \kappa\\
    1-t,&\text{ if }p=1,~q=(t,\alpha),\text{ for some }\alpha\in \kappa\\
    |t-s|,&\text{ if }p=(t,\alpha),~q=(s,\alpha),\text{ for some }\alpha\in \kappa\\
    \min\{d(p,0)+d(0,q),d(p,1)+d(1,q)\},&\text{ otherwise}.
    \end{cases}
    $$
    It is clear that $M$ is non-separable and is not isometric to an $\mathbb{R}$-tree. For every $\alpha\in\kappa$, write $M_\alpha =\{0,1\}\cup\{(t,\alpha)\colon t\in (0,1),~\alpha\in \kappa\}$. 
    
    Fix any $\alpha_0\in \kappa$, and define $\Gamma = \{0\}\cup\kappa$, with the order $\preceq$ given by
    \begin{align*}
        0\prec \alpha\qquad &\text{for every }\alpha\in \kappa\\
        \alpha\prec \beta\qquad &\text{for }\alpha,\beta\in\kappa\text{ if and only if }\alpha = \alpha_0\text{ and }\beta\neq\alpha_0. 
    \end{align*}
    
    Clearly $\left(\Gamma,\preceq\right)$ is a rooted, $\sigma$-complete tree with finite height. Define $R_{0}\colon M\rightarrow \{0\}$ as the trivial retraction, and $R_{\alpha_0}\colon M\rightarrow M_{\alpha_0}$ as the identity map on $M_{\alpha_0}$ and $R_{\alpha_0}(t,\alpha)=(t,\alpha_0)$ for every $t\in(0,1)$ and $\alpha\in \kappa\setminus{\alpha_0}$. Finally, for every $\alpha\in \kappa\setminus\{\alpha_0\}$, define $R_{\alpha}\colon M\rightarrow M_{\alpha_0}\cup M_{\alpha}$ as the identity on $M_{\alpha_0}\cup M_{\alpha}$, and $R_{\alpha_0}(t,\alpha)=(t,\alpha_0)$ for every $t\in(0,1)$ and $\alpha\in \kappa\setminus\{\alpha_0,\alpha\}$. 

    It is now routine to check that $\{R_s\}_{s\in\Gamma}$ is indeed a $1$-Lipschitz retractional tree. 
    
\end{example}

\section{Linear structure of Lipschitz-free spaces associated to Banach spaces}

As mentioned throughout this note, it is a standard application of the linearization property of Lipschitz-free spaces that if $M$ is a complete metric space and $R\colon M\rightarrow S$ is a Lipschitz retraction onto a subset $S$ of $M$, there exists a bounded linear projection $\widehat{R}\colon \mathcal{F}(M)\rightarrow \mathcal{F}(S)$ between the associated Lipschitz-free spaces. It is also not difficult to show that not all bounded and linear projections in Lipschitz-free spaces are obtained in this way: a bounded and linear projection $P\colon \mathcal{F}(M)\rightarrow \mathcal{F}(S)$ is of the form $P=\widehat{R}$ for some Lipschitz retraction $R\colon M\rightarrow S$ if and only if $P(\delta(p))\in \delta(S)$ for all $p\in M$. 

However, if $X$ is a Banach space, we are able to define the \emph{barycenter} map $\beta_X\colon \mathcal{F}(X)\rightarrow X$, which is a linear map with $\|\beta_X\|\leq 1$ such that $\beta_X(\delta(x))=x$ for all $x\in X$ (see e.g.: \cite{GodKal03}). By composing with the isometry $\delta$, we obtain that every bounded and linear projection $P\colon \mathcal{F}(X)\rightarrow \mathcal{F}(Y)$ where $Y$ is a linear subspace of $X$ induces a Lipschitz retraction $R=\beta_Y\circ P\circ \delta\colon X\rightarrow Y$ with $\|R\|_\text{Lip}=\|P\|$. 

If $\mathcal{F}(X)$ is $r$-Plichko, then it admits in particular an $r$-projectional skeleton $\{P_s\}_{s\in \Gamma}$. By Proposition 2.2 in \cite{HajQui22}, we may assume that $P_s(\mathcal{F}(X))=\mathcal{F}(Y_s)$, where $Y_s$ is a separable subspace of $X$ for every $s\in \Gamma$. Therefore, using the definition of $r$-projectional skeleton and the previous discussion we arrive at the following observation:

\begin{remark}
\label{Plichko_implies_LSRP}
If $\mathcal{F}(X)$ is $r$-Plichko for some $r\geq 1$, every separable subspace of $X$ is contained in a separable subspace which is an $r$-Lipschitz retract of $X$. 
\end{remark}

Recall that a Banach space has the \emph{Separable Complementation Property (SCP)} if every separable subspace of $X$ is contained in a separable subspace that is linearly complemented in $X$. We say that $X$ has the $r$-SCP for $r\geq 1$ if the norm of the projections of the separable superspaces can be uniformly bounded by $r$. Analogously, we say that a metric space $M$ has the \emph{Lipschitz Separable Retractional Property (Lipschitz SRP)} if every separable subset of $M$ is contained in a separable subset which is a Lipschitz retract of $M$. For $r\geq 1$, we say that $M$ has the $r$-Lipschitz SRP if the norm of the Lipschitz retractions onto the separable supersets can be uniformly bounded by $r$.

Let us note that in the metric space setting, Remark \ref{Plichko_implies_LSRP} fails in a strong sense, as there exists a metric space $M$ such that $\mathcal{F}(M)$ is $1$-Plichko, but with two specific points $\{0,1\}\subset M$ such that every separable set containing them is not a Lipschitz retract of $M$ (see Theorem 3.7 and Remark 3.8 in \cite{HajQui22}).

Remark \ref{Plichko_implies_LSRP} suggests that a strategy for finding a Lipschitz-free space without the Plichko property could be to construct a Banach space failing the Lipschitz SRP. This is a difficult problem, which implies the existence of a separable Banach space that is not a Lipschitz retraction of its bidual; a problem that goes back to the seminal paper \cite{Lin64} by Lindenstrauss, and whose non-separable version was solved by Kalton in \cite{Kal11}. Indeed, suppose there existed a Banach space $X$ and a separable subspace $Y$ which is not contained in any separable Lipschitz retract of $X$. By a classic result of Heinrich and Mankiewicz \cite{HeiMan82}, there exists a separable subspace $Z$ containing $Y$ which is locally complemented, i.e.: $Z^{**}$ is linearly complemented in $X^{**}$. Therefore, $Z$ is not a Lipschitz retract of $Z^{**}$, since this would imply that $Z$ is Lipschitz retract of $X$. 

We finish this note with a brief result on $\ell_\infty$. Although $\ell_\infty$ fails the SCP, it enjoys the $2$-Lipschitz SRP (this holds in general for all $C(K)$ spaces; see Proposition 2.11 in \cite{HajQui22}). Modifying the classical argument that shows that $c_0$ is not an $r$-Lipschitz retract of $\ell_\infty$ for $r<2$, we prove that the constant $2$ is also optimal for the Lipschitz SRP in $\ell_\infty$.

\begin{theorem}    
Let $S$ be a separable subset of $\ell_\infty$ containing $c_0$, and suppose there exists a Lipschitz retraction $R\colon \ell_\infty\rightarrow S$. Then $\|R\|_\text{Lip}\geq 2$. Consequently, $\mathcal{F}(\ell_\infty)$ is not $r$-Plichko for $r<2$. 
\end{theorem}
\begin{proof}
    Put $S =\overline{\{x_n=(x_n(k))_{k\in\mathbb{N}}\in \ell_\infty\colon n\in \mathbb{N}\}}$. Consider the sequence $z\in \ell_\infty$ given by
    $$z(k)= -\text{sign}(x_k(k))\qquad\text{ for }k\in\mathbb{N}. $$
    For every $k\in\mathbb{N}$, let $P_k\colon \ell_\infty\rightarrow \ell_\infty$ be the linear projection given by $\left(P_kx\right)(i)=x(i)$ if $i\leq k$ and $\left(P_kx\right)(i)=0$ if $i>k$ for every $x\in\ell_\infty$. With this notation, define for every $k\in\mathbb{N}$ the sequence $y_k = 2P_k(z)$, which belongs to $S$ since $S$ contains $c_0$. It is straightforward to check that $\|z-y_k\|_\infty =1$ for all $k\in\mathbb{N}$.

    For every $n\in\mathbb{N}$ and every $k\geq n$, it holds that $\|x_n-y_k\|_\infty\geq 2$, since $|y_k(n)|=2$ and the $n$-th coordinate of $x_n$ has the opposite sign. It follows that $\limsup_{k\rightarrow +\infty} \|x-y_k\|_\infty\geq 2$ for every $x\in S$. Therefore:
    $$ 2\leq \limsup_{k\rightarrow +\infty} \|R(z)-y_k\|_\infty\leq \limsup_{k\rightarrow +\infty}\|R\|_\text{Lip}\| z-y_k\|_\infty=\|R\|_\text{Lip},$$
    and the first claim is proven. The second part of the statement follows from Remark \ref{Plichko_implies_LSRP}.
\end{proof}

\section{Acknowledgements}

This research was supported by grant PID2021-122126NB-C33 funded by MCIN/AEI/\allowbreak10.13039/\allowbreak501100011033 and by “ERDF A way of making Europe”. The second author's research has been supported by PAID-01-19, by GA23-04776S and by the project SGS23/\allowbreak056/OHK3/\allowbreak 1T/13.

\printbibliography

\end{document}